\documentclass[12pt]{amsart}
\usepackage{amscd,amsmath,amsthm,amssymb,graphics}

\usepackage{pgf,tikz}
\usetikzlibrary{arrows}
\usepackage{lineno}

\def\NZQ{\Bbb}               

\def\ZZ{{\NZQ Z}}

\def\frk{\frak}               

\def\Phi{{\frk n}}
\def\Phi{{\frk N}}
\def\opn#1#2{\def#1{\operatorname{#2}}} 
\opn\chara{char} \opn\length{\ell} \opn\pd{pd} \opn\rk{rk}
\opn\projdim{proj\,dim} \opn\injdim{inj\,dim} \opn\rank{rank}
\opn\depth{depth} \opn\grade{grade} \opn\height{height}
\opn\embdim{emb\,dim} \opn\codim{codim}

\opn\Tr{Tr} \opn\bigrank{big\,rank}
\opn\superheight{superheight}\opn\lcm{lcm}
\opn\trdeg{tr\,deg}
\opn\reg{reg} \opn\lreg{lreg} \opn\ini{in} \opn\lpd{lpd}
\opn\size{size}\opn\bigsize{bigsize}
\opn\cosize{cosize}\opn\bigcosize{bigcosize}
\opn\sdepth{sdepth}\opn\sreg{sreg}
\opn\link{link}\opn\fdepth{fdepth}
\opn\index{index}
\opn\index{index}
\opn\indeg{indeg}
\opn\N{N}
\opn\SSC{SSC}
\opn\SC{SC}
\opn\lk{lk}
\opn\div{div} \opn\Div{Div} \opn\cl{cl} \opn\Cl{Cl}
\opn\Spec{Spec} \opn\Supp{Supp} \opn\supp{supp} \opn\Sing{Sing}
\opn\Ass{Ass} \opn\Min{Min}\opn\Mon{Mon} \opn\dstab{dstab} \opn\astab{astab}
\opn\Syz{Syz}
\opn\reg{reg}
\opn\Ann{Ann} \opn\Rad{Rad} \opn\Soc{Soc}
\opn\Im{Im} \opn\Ker{Ker} \opn\Coker{Coker} \opn\Am{Am}
\opn\Hom{Hom} \opn\Tor{Tor} \opn\Ext{Ext} \opn\End{End}\opn\Der{Der}
\opn\Aut{Aut} \opn\id{id}

\opn\nat{nat}
\opn\pff{pf}
\opn\Pf{Pf} \opn\GL{GL} \opn\SL{SL} \opn\mod{mod} \opn\ord{ord}
\opn\Gin{Gin} \opn\Hilb{Hilb}\opn\sort{sort}
\opn\initial{init}
\opn\ende{end}
\opn\height{height}
\opn\type{type}
\opn\aff{aff} \opn\con{conv} \opn\relint{relint} \opn\st{st}
\opn\lk{lk} \opn\cn{cn} \opn\core{core} \opn\vol{vol}
\opn\link{link} \opn\star{star}\opn\lex{lex}
\opn\gr{gr}

\def\pot#1#2{#1[\kern-0.28ex[#2]\kern-0.28ex]}

\opn\dirlim{\underrightarrow{\lim}}
\opn\inivlim{\underleftarrow{\lim}}
\def\Implies{\ifmmode\Longrightarrow \else
        \unskip${}\Longrightarrow{}$\ignorespaces\fi}
\def\implies{\ifmmode\Rightarrow \else
        \unskip${}\Rightarrow{}$\ignorespaces\fi}
\def\iff{\ifmmode\Longleftrightarrow \else
        \unskip${}\Longleftrightarrow{}$\ignorespaces\fi}

\let\:=\colon
\newtheorem{Theorem}{Theorem}[section]
 \newtheorem{Lemma}[Theorem]{Lemma}
 \newtheorem{Corollary}[Theorem]{Corollary}
 
 \newtheorem{Remark}[Theorem]{Remark}
 
 \newtheorem{Example}[Theorem]{Example}
 
 \newtheorem{Definition}[Theorem]{Definition}

\let\epsilon\varepsilon
\let\kappa=\varkappa
\def\qed{\ifhmode\textqed\fi
      \ifmmode\ifinner\quad\qedsymbol\else\dispqed\fi\fi}
\def\textqed{\unskip\nobreak\penalty50
       \hskip2em\hbox{}\nobreak\hfil\qedsymbol
       \parfillskip=0pt \finalhyphendemerits=0}
\def\dispqed{\rlap{\qquad\qedsymbol}}

\opn\dis{dis}
\def\pnt{{\raise0.5mm\hbox{\large\bf.}}}

\opn\Lex{Lex}

\begin{document}

\title{On  reduction numbers of products of  ideals
 }
\author{ Dancheng Lu} \author{ Tongsuo Wu}

\address{Dancheng Lu, School of  Mathematical Sciences, Soochow University, 215006 Suzhou, P.R.China}
\email{ludancheng@suda.edu.cn}

\address{Tongsuo Wu, School of  Mathematical Sciences, Shanghai Jiaotong University, 210030 Shanghai, P.R.China}
\email{tswu@sjtu.edu.cn}

\thanks{Corresponding author: Dancheng Lu, School of  Mathematical Sciences, Soochow University, P.R.China, ludancheng@suda.edu.cn }



\begin{abstract} Let $R$ be a standard graded  algebra over an infinite  field $\mathbb{K}$ and $M$  a finitely generated $\ZZ$-graded $R$-module.
Let $I_1,\ldots I_m$ be graded ideals of $R$. The functions $r(M/I_1^{a_1}\ldots I_m^{a_m}M)$ and $r(I_1^{a_1}\ldots I_m^{a_m}M)$ are investigated and their asymptotical behaviours are given.  Here $r(\bullet)$  stands for the reduction number of a finitely generated graded $R$-module $\bullet$.
\end{abstract}

\maketitle

\section{Introduction}

 In this short note, we always assume that $R=\bigoplus_{n\geq 0}R_n$ is a standard graded Noetherian algebra over an infinite field $\mathbb{K}$. Here ``standard graded" means that $R_0=\mathbb{K}$ and $R=\mathbb{K}[R_1]$. As usual, a nonzero element in $R_1$ is called a {\it linear form} of $R$. Let $M$ be a  finitely generated nonzero $\ZZ$-graded $R$-module. Recall the following definition from \cite{L}.

\begin{Definition}\label{dreduction}{\em A graded ideal $J$ of $R$ is called an {\it $M$-reduction}  if $J$ is an ideal generated by linear forms  such that $(JM)_n=M_n$ for $n\gg0$; An $M$-reduction  is called {\it minimal} if  it does not contain any other $M$-reduction. The {\it reduction number} of $M$ with respect to $J$ is defined to be $$r_J(M)=\max\{n\in \ZZ\:\; (JM)_n\neq M_n\}.$$ The {\it reduction number} of $M$ is $$r(M)=\min\{r_J(M)\:\; J \mbox{ is a minimal } M\mbox{-reduction}\}.$$
}\end{Definition}

Let  $I$ be a graded  ideal of $R$.  We show in \cite{L} that the functions $r(M/I^nM)$ and $r(I^nM)$ are both linear functions for  $n\gg 0$, and moreover, the latter one has the same slope as the  function $\mathrm{reg} (I^nM)$, which is also asymptotically linear \cite{CHT,K,TW}. On the contrast, if  we let $I_1,\ldots, I_m$ be  graded  ideals of $R$,  $r(I_1^{a_1}\cdots I_m^{a_m}M)$ needs not be  a linear function of $(a_1,\ldots, a_m)$  asymptotically when $m\geq 2$, as shown by \cite[remark 2.8]{L}. The goal of this  note is to investigate the asymptotically properties of  multi-functions $r(M/I_1^{a_1}\ldots I_m^{a_m}M)$ and $r(I_1^{a_1}\ldots I_m^{a_m}M)$ more carefully.

To simplify the exposition we set $$I^{\mathbf{a}}=I_1^{a_1}\cdots I_m^{a_m} \mbox{\quad for any } \mathbf{a}=(a_1,\ldots, a_m)\in \mathbb{N}^{m}. $$

 Let  ``$\leq $" be the natural partial order on $\mathbb{Z}^m$, namely: $\mathbf{a}\leq \mathbf{b }\Longleftrightarrow  \mathbf{b}-\mathbf{a}\in \mathbb{N}^m, \mbox{ } \forall \mbox{ } \mathbf{a},\mathbf{b}\in \mathbb{Z}^m.$  We say that a formula (or a property) holds for $\mathbf{a}\gg \mathbf{0}$ if there exists a vector $\mathbf{b}\in \mathbb{N}^m$ such that it holds for all $\mathbf{a}\geq \mathbf{b}$.

To state our main results conveniently, we give the following definition.

\begin{Definition} {\em Given a set $D$  of positive integers $d_{ij}$, indexed by  $1\leq i\leq m, 1\leq j\leq g_i $, where $m, g_i, i=1,\ldots,m$ are some positive integers.   By} a linear function $L$  with coefficients in $D$ {\em (resp. in $D\cup\{0\})$)  we mean  a linear function on $\mathbb{Z}^m$: $$L(\mathbf{a})=\sum_{i=1}^m\lambda_{i}a_i+\lambda_{0}$$ satisfying  $$\lambda_{i}\in \{d_{i1},\ldots,d_{ig_i}\} \mbox{\quad (resp.  } \lambda_{i}\in \{0, d_{i1},\ldots,d_{ig_i}\} ) $$ for  $ i=1,\ldots,m$. In this case, we also say the linear function } $L$ has coefficients in $D$ {\em (resp. in $D\cup \{0\}$}).
\end{Definition}

Let $f_{i1},\ldots, f_{ig_i}$ be a minimal system of generators of  the ideal $I_i$ for $i=1,\ldots, m$. Denote by $d_{ij}$ the degree of $f_{ij}$ for all possible $i,j$, and let $D$ be the indexed set $\{d_{ij} : 1\leq i\leq m, 1\leq j\leq g_i\}$.  Our main results   show that the function $r(I^{\mathbf{a}}M)$ is given by the supremum of finitely many linear functions with coefficients in $D$ for $\mathbf{a}\gg \mathbf{0}$, while the function $r(M/I^{\mathbf{a}}M)$ is given by the supremum of finitely many linear functions with coefficients in $D\cup \{0\}$ for $\mathbf{a}\gg \mathbf{0}$. Moreover, at least one of these linear functions has coefficients in $D$.  In addition, some corollaries concerning the special cases when $m=1$ or when $d_{i1}=\cdots=d_{i,g_i}$ for $i=1,\ldots,m$ are given.

\section{A generalization of an old theorem }

A result due to Trung and Wang, i.e., \cite[Theorem 2.2]{TW}, plays an important role in the proof of main results of \cite{L}. Denote by  $S:=A[x_1,\ldots,x_n,y_1,\ldots,y_m]$ the polynomial ring over a commutative Noetherian ring $A$ with unity. We may view $S$ as a bi-graded ring with $\deg x_i=(1,0),i=1,\ldots, n$ and $\deg y_j=(d_j,1), j=1,\dots, m,$ where $d_1,\ldots,d_m$ is a sequence of non-negative integers.  Let $\mathcal{U}$ be a finitely generated bi-graded module over $S$. For a fixed number $n$ put
$\mathcal{U}_n:=\bigoplus_{a\in \mathbb{Z}}\mathcal{U}_{(a,n)}.$ Clearly,  $\mathcal{U}_n$ is a finitely generated graded module over the naturally graded ring  $A[x_1,\ldots,x_n]$.  It was proved in \cite[Theorem 2.2]{TW}  that $\mathrm{reg} (\mathcal{U}_n)$ is asymptotically a linear function of $n$ with slope $\leq \max\{d_1,\ldots,d_m\}.$

In this section, we will extend this theorem to the case when $S$ and $\mathcal{U}$ are  multi-graded. For this, we  let $S=K[x_1,\ldots,x_n, z_{ij}, 1\leq i\leq m, 1\leq j\leq g_i]$ be the polynomial ring with grading $\deg x_i=(1,\mathbf{0})$ for $i=1,\ldots,n$ and $\deg z_{ij}=(d_{ij}, \mathbf{e}_i)$ for $1\leq i\leq m, 1\leq j\leq g_i$, where $\mathbf{0}$ is the zero element of $\mathbb{Z}^m$ and $\mathbf{e}_1,\ldots,\mathbf{e}_m$ is the standard basis of $\mathbb{Z}^m$. Let $\mathcal{U}$ be the finitely generated $\mathbb{Z}\times \mathbb{Z}^m$-graded $S$-module. For any $\mathbf{a}\in Z^m$, define $$\mathcal{U}_{\mathbf{a}}=\bigoplus_{i\in \mathbb{Z}}\mathcal{U}_{(i,\mathbf{a})}.$$  Then $\mathcal{U}_{\mathbf{a}}$  is a finitely generated graded $R$-module for all $\mathbf{a}\in \mathbb{Z}^m$,
where $R$ is the polynomial ring $\mathbb{K}[x_1,\ldots,x_n]$ with grading $\deg x_i=1$ for $i=1,\ldots,n$.

The regularity of a finitely generated graded $R$-module $M$ can be defined in two ways. The first one (also called {\it absolute regularity} ) is based on  local cohomology. It is defined by $$\reg(M)=\sup \{i+j: H_{\mathfrak{m}}^i(M)_j\neq 0\},$$
where $\mathfrak{m}$ is the ideal $(x_1,\ldots,x_n)$ of $R$. The other one is defined by Betti numbers: $$\reg(M)=\sup \{j-i: \beta_{i,j}(M)\neq 0\}.$$
On the other hand, Betti numbers can be computed by the Koszul homology: $$\beta_{i,j}(M)=\dim_{\mathbb{K}} H_i(\underline{x},M)_j,$$ where $\underline{x}$  is the sequence $x_1,\ldots,x_n$.  Hence, for all $\mathbf{a}\in \mathbb{Z}^m$,  one has $$\beta_{i,j}(\mathcal{U}_{\mathbf{a}})=\dim_{\mathbb{K}} H_i(\underline{x},\mathcal{U})_{(j,\mathbf{a})}.$$
We refer readers to books \cite{BH} or \cite{HH} for the basic knowledge in this direction.

\begin{Theorem} \label{g}  Let $S,\mathcal{U},R$ be as above. There exist linear functions $L_1,\ldots, L_c$  with coefficients in $D$
 such that $$\mathrm{reg}(\mathcal{U}_{\mathbf{a}})=\max\{L_k(\mathbf{a}): 1\leq k\leq c\} \mbox{\quad for all \quad } \mathbf{a}\gg \mathbf{0}.$$

\end{Theorem}

\begin{proof}

 Since the finitely generated $S$-module $H_i(\underline{x},\mathcal{U})$ is annihilated by $(\underline{x})$, it is also a finitely generated $\mathbb{Z}\times \mathbb{Z}^m$-graded module over $\mathbb{K}[z_{ij}, 1\leq i\leq m, 1\leq j\leq g_i]$. Set $$\rho_{i}(\mathbf{a})=\max\{j: H_i(\underline{x},\mathcal{U})_{(j,\mathbf{a})}\neq 0\}.$$  It follows by \cite[Lemma 2.1]{BC} that $\rho_{i}(\mathbf{a})$ is the supermum of some linear functions with coefficients in $D$ for $\mathbf{a}\gg \mathbf{0}$.
Therefore, $\reg (\mathcal{U}_{\mathbf{a}})=\max\{\rho_{i}(\mathbf{a})-i: i\geq 0\}$ is also the supermum of some linear functions with coefficients in $D$ for $\mathbf{a}\gg \mathbf{0}$, as required.
\end{proof}

\begin{Remark} \em The idea of  Theorem~\ref{g} comes from \cite{BC}, we just state it explicitly.
\end{Remark}

 Comparing with \cite[Theorem 2.2]{TW}, Theorem~\ref{g} has some advantages.   For one thing, it considers the multi-garded case, which is more general than  the bi-graded case. For another thing, it asserts  that  all the coefficients of the involved linear functions  belong to the set of the first coordinates of multi-degrees of variables, while according to \cite[Theorem 2.2]{TW}, one only knows that the coefficient is  less than or equal to  the maximum of the first coordinates of these multi-degrees.

\section{Asymptotic behaviours}
Let $I_1,\ldots, I_m$ be graded ideals of a standard graded algebra $R$ and $M$ a finitely generated graded $R$-module.
In this section, we will investigate the asymptotic behaviours of multi-functions
$r(I_1^{a_1}\cdots I_m^{a_m}M)$ and $r(M/I_1^{a_1}\cdots I_m^{a_m}M)$.

The reduction number of a graded module is independent of its base rings. More precisely, if the standard graded algebra $R$ is a homomorphic image of the polynomial ring $\mathbb{K}[x_1,\ldots,x_n]$ with the same embedding dimension, then for any finitely generated graded $R$-module $M$,  the reduction number of $M$ as an $R$-module is the same as the reduction number of $M$ as a $\mathbb{K}[x_1,\ldots,x_n]$-module. In the sequel we will assume $R$ is the polynomial ring $\mathbb{K}[x_1,\ldots,x_n]$ with grading $\deg x_i=1$ for $i=1,\ldots, n$.

In the following results, we will also use $I^{\mathbf{a}}$ to denote the product $I_1^{a_1}\cdots I_m^{a_m}$. By convention,  $I^{\mathbf{a}}=R$ if $\mathbf{a}\in \mathbb{Z}^m\setminus \mathbb{N}^m$. As said before,   $I_i, i=1,\ldots,m$ are graded ideals of $R$ and they are  generated in degrees $d_{i1},\ldots,d_{ig_i}, i=1,\ldots,m$ respectively. The indexed set $\{d_{ij}: 1\leq i\leq m, 1\leq j\leq g_i \}$ is  denoted by $D$ again. Let $M$ be a finitely generated $R$-module. We use $\mathrm{a}(M)$ to denote the largest degree of non-vanishing components, that is, $$\mathrm{a}(M)=\max\{n\in \mathbb{Z}: M_n\neq 0\}.$$
It is known that if $\dim (M)=0$,  then $\reg (M)=\mathrm{a}(M)$. Here and hereafter $\dim(\bullet)$ stands for the Krull dimension of a module $\bullet$.

\begin{Lemma}\label{d1}  The dimension function $\dim (I^{\mathbf{a}}M)$ is stationary for $\mathbf{a}\gg\mathbf{0}$.
\end{Lemma}
\begin{proof} This is because this function is non-increasing and always non-negative.
\end{proof}

\begin{Theorem} \label{T1} There exist linear functions $L_1,\ldots, L_c$ on $\mathbb{Z}^m$ with coefficients in $D$
such that $$r(I^{\mathbf{a}}M)=\max\{L_k(\mathbf{a}): 1\leq k\leq m\} \mbox{\quad for all \quad } \mathbf{a}\gg \mathbf{0}.$$
\end{Theorem}
\begin{proof} This result can follow in the same way as \cite[Theorem 2.7]{L} by Lemma ~\ref{d1} and by replacing  \cite[Theorem 2.2]{TW} with Theorem~\ref{g}.
\end{proof}

\begin{Lemma} \label{d2} The dimension function $\dim (M/I^{\mathbf{a}}M)$ is stationary for $\mathbf{a}\geq \mathbf{1}$, where $\mathbf{1}$ is the vector $(1,\ldots,1)\in \mathbb{N}^m.$

\end{Lemma}

\begin{proof} This follows from the equality $\sqrt{\mbox{Ann}(M/I^{\mathbf{a}}M)}=\sqrt{\mathrm{Ann}(M)+I_1\cdots I_m}$ for any $\mathbf{a}\geq \mathbf{1}$.
\end{proof}

\begin{Theorem} \label{T2} There exist  linear functions
$L_{k}, 1\leq k\leq c$, with coefficients in $D\cup \{0\}$
such that $$r(M/I^{\mathbf{a}}M)=\max\{L_k(\mathbf{a}):1\leq k\leq c\}$$
for all $\mathbf{a}\gg \mathbf{0}$. Moreover, at least one of these linear functions has coefficients in $D$.

\end{Theorem}

\begin{proof} In view of \cite[remark 2.4]{L} and by Lemma~\ref{d2}, we may assume that there exists a graded ideal $J$ of $R$ such that $$r(M/I^{\mathbf{a}}M)=r_J(M/I^{\mathbf{a}}M)=\mathrm{a}(M/JM+I^{\mathbf{a}}M)$$ for all  $\mathbf{1}\leq \mathbf{a}\in \mathbb{N}^m,$
where the last equality follows from Definition~\ref{dreduction}.

Let $\mathbf{e}_1,\ldots,\mathbf{e}_m$ be the standard basis of $\mathbb{Z}^m$. For any $\mathbf{a}\in \mathbb{N}^m$ and $1\leq i\leq m$, we set $$\mathcal{U}_{\mathbf{a}}^i=I^{\mathbf{a}-\mathbf{e_i}}M+JM/I^{\mathbf{a}}M+JM$$ and set $$\mathcal{U}^i=\bigoplus_{\mathbf{a}\in \mathbb{N}^m} \mathcal{U}_{\mathbf{a}}^i.$$   Then $\mathcal{U}^i$ is a $\mathbb{Z}\times \mathbb{Z}^m$-graded $S$-module  for   $i=1,\ldots,m$, where $S$ is the multi-graded polynomial ring $K[x_1,\ldots,x_n, z_{ij}, 1\leq i\leq m, 1\leq j\leq g_i]$, as defined in Section 2. By Theorem~\ref{g} there exists a vector $\mathbf{1}\leq \mathbf{v}=(v_1,\ldots,v_m)\in \mathbb{N}^m$ and linear functions $L_{ik}(\mathbf{a}), 1\leq i\leq m, 1\leq k\leq c_i$ with coefficients in $D$ such that for all $1\leq i\leq m$ and for all $\mathbf{a}\geq \mathbf{v}$, one has
\begin{equation} \label{(1)} \mathrm{a}(\mathcal{U}_{\mathbf{a}}^i)=\reg (\mathcal{U}_{\mathbf{a}}^i)=\max\{L_{ik}(\mathbf{a}): 1\leq k\leq c_i\}.
 \end{equation}

 On the other hand, from the short exact sequences $$0\rightarrow \mathcal{U}^i_{\mathbf{b}}\rightarrow M/I^{\mathbf{b}}M+JM\rightarrow M/I^{\mathbf{b}-\mathbf{e}_i}M+JM\rightarrow 0,$$ where $\mathbf{b}\in \mathbb{N}^m$ and $1\leq i\leq m$,  it follows that
 \begin{equation} \label{t}  r(M/I^{\mathbf{b}}M)=\max\{r(M/I^{\mathbf{b}-\mathbf{e}_i}M), \mathrm{a}(\mathcal{U}^i_{\mathbf{b}})\}
 \end{equation}
 for any $\mathbf{b}\in \mathbb{N}^m$ and $1\leq i\leq m$ with $\mathbf{b}-\mathbf{e}_i\geq \mathbf{1}.$

 Using the equality  (\ref{t}) repeatedly with suitable $\mathbf{b}$'s and $i$'s, we obtain:
 \begin{equation}\label{(2)}
 \begin{aligned}
 r(M/I^{\mathbf{a}}M) &=\max\{r(M/I^{\mathbf{v}}M), \mathrm{a}(\mathcal{U}_{\mathbf{v}+\mathbf{e}_1}^1),\ldots,  \mathrm{a}(\mathcal{U}_{\mathbf{v}+(a_1-v_1)\mathbf{e_1}}^1), \ldots, \mathrm{a}(\mathcal{U}_{\mathbf{a}}^m)  \}\\
   &=\max\{\mathrm{a}(\mathcal{U}_{\mathbf{v}+\mathbf{e}_1}^1),\ldots,  \mathrm{a}(\mathcal{U}_{\mathbf{v}+(a_1-v_1)\mathbf{e_1}}^1), \ldots, \mathrm{a}(\mathcal{U}_{\mathbf{a}}^m) \}
 \end{aligned}
  \end{equation}
  for all $\mathbf{a}\gg \mathbf{0}$

  Combining  (\ref{(1)}) with (\ref{(2)}), it follows that $$r(I^{\mathbf{a}}M)=\max\{L_{ik}(a_1,\ldots,a_i, v_{i+1},\ldots,v_m): 1\leq i\leq m, 1\leq k\leq c_i\}$$ for all $\mathbf{a}\gg \mathbf{0}.$
  Note that $L_{ik}(a_1,\ldots,a_i, v_{i+1},\ldots,v_m)$ is a linear function of $\mathbf{a}$ with coefficients in $D\cup \{0\}$
  for each $i,k$, our desired equality follows.

  Finally,   the last statement of this theorem  follows since $L_{mk}(a_1,\ldots,a_m)$  are  linear functions of $\mathbf{a}$ with coefficients in $D$ for $1\leq k\leq c_k$.
\end{proof}

Restricting  Theorems~\ref{T1} and \ref{T2} to the case when $m=1$, we obtain the following result.

\begin{Corollary}

  Let $I$ be a graded ideal of $R$ generated in degrees $d_1,\ldots,d_r$. Then $r(I^nM)$  and $r(M/I^nM)$ are both linear functions for $n\gg 0$ whose slopes belong to $\{d_1,\ldots,d_r\}$.

\end{Corollary}

\begin{Remark}  \em According to \cite[Theorem 2.7]{L},  $r(I^nM)$ is asymptotically a linear function   with $\rho_I(M)$ as its slope. Thus, $\rho_I(M)\in \{d_1,\ldots, d_r\}$. Here $\rho_I(M)$ is defined to be
$$\rho_I(M)=\min\{ D(J): JI^{n-1}M=I^nM \mbox{ and  } J\subseteq I  \mbox{ for some } n>0  \},$$ where $D(J)$ is the maximal degree of minimal generators of a graded ideal $J$.

\end{Remark}

The following example shows that the slopes of $r(I^nM)$ and $r(M/I^nM)$ could be any of the numbers  $d_1,\ldots,d_r$.

\begin{Example} \em Let $I$  be the ideal  $(x_1^{d_1},\ldots,x_r^{d_r})$ of $R=K[x_1,\ldots,x_r]$ and set $$M_i=R/(x_1,\ldots, x_{i-1}, x_{i+1},\ldots,x_r)\mbox{\quad for \quad} i=1,\ldots, r.$$ Then $$I^nM_i=\frac{x_i^{nd_i}+(x_1,\ldots, x_{i-1}, x_{i+1},\ldots,x_r)}{(x_1,\ldots, x_{i-1}, x_{i+1},\ldots,x_r)},$$
and $$M_i/I^nM_i\cong \mathbb{K}[x_i]/(x_i^{nd_i}).$$
From these it follows that both $r(I^nM_i)$ and $r(M_i/I^nM_i)$ have $d_i$ as their slopes. Here, we use the fact that if $M$ is a graded module of dimension zero then the zero ideal is the unique minimal $M$-reduction.
\end{Example}
Another immediate consequence of Theorems~\ref{T1} and \ref{T2} is the following.

\begin{Corollary}
  If, for each $i\in \{1,\ldots,m\}$,  $I_i$ is generated in a single degree, say $d_i$,  then there exist integers $\epsilon_1$ and $\epsilon_2$ such that $r(I^{\mathbf{a}}M)=d_1a_1+\cdots+d_ma_m +\epsilon_1$ and $r(M/I^{\mathbf{a}}M)=d_1a_1+\cdots+d_ma_m +\epsilon_2$ for all $\mathbf{a}\gg \mathbf{0}$.
\end{Corollary}

{\noindent \bf Acknowledge:}  We would like to express our thanks to the referee for his/her careful reading and good advices. This project receives the support of NSFC (No. 11971338) and NSF of Shanghai (No. 19ZR1424100).


\begin{thebibliography}{}

\bibitem{BC} W. Bruns, A. Conca, \textit{A remark of regularity of powers and products of ideals}, J. Pure Appl. Algebra  221 (2017), 2801-2808.


\bibitem{BH} W. Bruns, J. Herzog, \textit{Cohen-Macaulay rings}, Cambridges  Studies in Advanced Mathematics, 39, 1993.
\bibitem{CHT} S.D. Cutkosky, J. Herzog, N.V. Trung, \textit{Asymptotic behaviour of the Castelnuovo-Mumford regularity},
Compositio Math. 118 (1999), 243-261.


\bibitem{HH}
J.~Herzog,  T.~Hibi, \textit{Monomial Ideals}, Graduate Text in Mathematics, Springer, 2011.




\bibitem{L} D.C. Lu, \textit{On the asymptotic linearity of reduction number}, J. Algebra  504 (2018),  1--9.



\bibitem{K} V. Kodiyalam,  \textit{Asymptotic behaviour of Castelnuovo-Mumford regularity}, Proc. A.M.S. 128 (2000), 407-411.



\bibitem{TW}  N.V. Trung, Hsin-Ju Wang, \textit{On the asymptotic linearity of
Castelnuovo-Mumford regularity}, J. Pure Appl. Algebra 201 (2005), 42--48.




\end{thebibliography}
\end{document}